\documentclass[11pt, reqno]{amsart}

\evensidemargin
\oddsidemargin 
\makeatletter\@addtoreset{equation}{section}\makeatother

\usepackage{amsmath, amsthm, epsfig, latexsym, amssymb, bbm}
\usepackage[dvips]{color}
\usepackage{color,calc}

\makeatletter
\newenvironment{shadedbox}[1][]%
        {
          \raggedright
        \setlength{\rightmargin}{\leftmargin}
        \setlength{\itemsep}{-12pt}
        \setlength{\parsep}{20pt}
        \begin{lrbox}{\@tempboxa}%
        \begin{minipage}{\linewidth-2\fboxsep}
        }%
        {
        \end{minipage}%
        \end{lrbox}%
        \fbox{\usebox{\@tempboxa}}\newline
        }%
\makeatother

\usepackage{pstricks,pst-node,pst-text,pst-3d}
\usepackage{psboxit}

\newcommand{\Kl}{\left(}
\newcommand{\Kr}{\right)}

\newcommand{\eps}{\varepsilon}
\newcommand{\emm}{\mathfrak m}
\newcommand{\del}{\delta}
\newcommand{\ssup}[1] {{\scriptscriptstyle{({#1}})}}
\newcommand{\one}{{\mathbbm 1}}

\newcommand{\E}{\mathbb E}
\renewcommand{\P}{\mathbb P}
\newcommand{\bP}{\mathbf{P}}

\newcommand{\cG}{\mathcal{G}}

\renewcommand{\phi}{\varphi}

\newcommand{\N}{\mathbb N}

\renewcommand{\P}{\mathbb P}

\newtheorem{rem}{Remark}
\newtheorem{theorem}{Theorem}

\newtheorem{lemma}[theorem]{Lemma}

\newtheorem{prop}[theorem]{Proposition}

\theoremstyle{remark}

\newtheorem{example}{Example}

\newcommand{\heap}[2]  {\genfrac{}{}{0pt}{}{#1}{#2}}
\newcommand{\sfrac}[2] {\mbox{$\frac{#1}{#2}$}}

\def\1{{\mathchoice {1\mskip-4mu\mathrm l}      % Blackboard bold 1
{1\mskip-4mu\mathrm l}
{1\mskip-4.5mu\mathrm l} {1\mskip-5mu\mathrm l}}}

\newcommand{\ind}{1\hspace{-0.098cm}\mathrm{l}}

\begin{document}

%\tableofcontents

\title[Typical distances in ultrasmall preferential attachment graphs]
{Typical distances in ultrasmall random networks}

\author[Steffen Dereich, Christian M\"onch, and Peter M\"orters]{Steffen Dereich, Christian M\"onch, and Peter M\"orters}

\maketitle

\vspace{-0.1cm}

\begin{quote}
{\small {\bf Abstract:} }
We show that in preferential attachment models with power-law exponent $\tau\in(2,3)$ the distance between
randomly chosen vertices in the giant component is asymptotically equal to %equal to 
$(4+o(1))\, \frac{\log\log N}{-\log (\tau-2)}$, where $N$ denotes the number of nodes. This is 
twice the value obtained for several types of configuration models with the same power-law exponent. 
%, random graphs with given expected degree, %of Chung and Lu. 
%and conditionally Poissonian graphs. % of Norros and Reittu.
%We explain this phenomenon by the 
The extra factor reveals the different structure of typical 
shortest paths in preferential attachment graphs. 
\end{quote}

%\vspace{0.1cm}

{\footnotesize
\vspace{0.1cm}
\noindent\emph{MSc Classification:}  Primary 05C80 Secondary 60C05, 90B15.

}

\vspace{0.1cm}

\section{Introduction}

One of the central observations in the theory of scale-free random networks is that in the case of power-law exponents
$\tau\in(2,3)$ networks are \emph{ultrasmall}, which means that the distance of two randomly chosen nodes in the giant component of a
graph with $N$ vertices is of asymptotic order $\log\log N$. The first analytical, but mathematically nonrigorous, evidence for 
this general phenomenon can be found, for example, in  Cohen and Havlin~\cite{CH03} or Dorogovtsev et al.~\cite{DMS03}, and 
there are also some early papers with rigorous results for specific network models, in particular the work of Reittu and Norros~\cite{RN02}
and the work of Chung and Lu~\cite{CL02}.
\medskip

In the present paper we refine this observation and identify graph distances including constant factors. Our main result is a universal technique for 
proving lower bounds for typical distances, which in a wide range of examples matches the best upper bounds known from the recent literature. 
The result is presented in the form of two theorems, which reveal that ultrasmall networks can be divided into two different \emph{universality classes}: 
For the class of ultrasmall preferential attachment models the typical distances turn out to be 
twice as large as for models in the class of configuration models. This difference corresponds to different structures 
of typical shortest paths in the network. We show that the two classes can be easily identified from the form 
of the attachment probability densities in the networks. We remark here that our work is focused on \emph{typical distances} in networks, as 
results on diameters tend to be model dependent and universality results are not to be expected. 
% see for example van der Hofstad and Hooghiemstra~\cite{HH08} for a discussion. 
\medskip%

At least informally, we have some structural insight into typical shortest paths in ultrasmall networks, see
for example Norros and Reittu~\cite{NR08}. For the class of \emph{configuration models} it turns out that typical vertices in the 
giant component can be connected with a few steps to a \emph{core} of the network. 
Within this core there is a hierarchy of \emph{layers} of nodes with increasing 
connectivity and at the top a small \emph{inner core} of highly connected nodes with very small diameter. 
A typical shortest path inside the core runs from one layer to the next until the inner core is reached, and 
then climbing down again until a  vertex in the lowest layer of the core is again connected to a typical vertex. 
\medskip

%\pagebreak[3]

A high degree of a vertex increases its connectivity to any other vertex, and hence
the layers can be identified by vertex  degrees. Very roughly speaking the $j$th layer consists of vertices with degree $k_j$ where $$\log k_j \approx (\tau-2)^{-j}$$ and there are about 
$$\frac{\log\log N}{-\log (\tau-2)}$$ layers. The graph distance of two randomly chosen vertices in the giant component 
is therefore $$\big(2+o(1)\big)\,  \frac{\log\log N}{-\log (\tau-2)}.$$ 
These asymptotics are rigorously confirmed for two variants of an inhomogeneous random graph model, by Chung and Lu~\cite{CL02} and Norros and 
Reittu~\cite{NR06}, and for a model with fixed degree sequence by van der Hofstad et al.\ in~\cite{HHZ07}. See also van der Hofstad~\cite{H11}
for a summary of various results with detailed proofs. In general upper bounds on the distances can be obtained by verifying the above strategy, while our Theorem~\ref{thm:noPA} provides a flexible (i.e. model-independent) approach to the lower bound. 
\medskip

For the more complex class of ultrasmall \emph{preferential attachment models} %with power-law exponents $\tau\in(2,3)$, 
existing results are far less complete. Dommers et al.~\cite{DHH10} show that for various ultrasmall preferential
attachment models %(defined in detail below) 
the typical distance of two vertices in the giant component is bounded
from above by $$\big(4+o(1)\big)\,  \frac{\log\log N}{-\log (\tau-2)}.$$ 
A corresponding lower bound, and hence confirmation of the exact factor 4, is stated as an interesting open problem 
by van der Hofstad and Hooghiemstra in \cite[IV.B]{HH08}
and again in~\cite{DHH10}, see the remark following Theorem~1.7 and Section~1.2.  Our main result, Theorem~1, provides this bound and 
confirms, somewhat surprisingly, that the upper bound is sharp. 
Besides the models given in~\cite{DHH10} we will  also describe other examples of random network models 
in the same universality class, in which Theorem~1 applies.%
\medskip%

Loosely speaking, the shortest paths in the class of preferential attachment models can be described as follows: Again, inside a core of highly connected vertices
paths run from bottom to top and back through a hierarchy of layers defined as before. However, by construction of the preferential attachment models a high degree 
of a vertex does not increase its connectivity to  \emph{all} vertices but only to those introduced \emph{late} into the system (which are typically outside the core). 
Therefore a path cannot directly connect one layer to another in one step, but it requires \emph{two steps}: The paths run from one layer to a young vertex 
and from there back into the next higher layer. The distance of two typical vertices is therefore increased by a factor of~two.
\medskip

In the following section we formulate the precise results, consisting of two simple hypotheses on a random network leading to the two different lower bound 
results, see Theorems~\ref{thm:PA} and~\ref{thm:noPA}. The section also contains a brief sketch of the proof technique and introduces the notation used in the proofs.  
In Section~3 we then discuss several examples of networks in the two universality classes. In all these 
examples upper bounds can either be found in the literature or derived by simple modifications of these proofs. Section~4 is devoted to the proofs of our main results.
%\medskip

\section{Main results}

%A \emph{dynamic graph model} is a sequence of unoriented, loop free 
%random graphs $(\cG_N)_{N\in\N}$ each having vertex labels $V_N=[N]$ and a random symmetric set of links $E_N\subset [N]\times [N]$. 
%Our aim is to analyse the (random) graph distance $d_N$ in $\cG_N$ of a pair of two uniformly chosen vertices $V,W\in[N]$. 
%In this article we analyse the regime with power law exponent $\tau\in(2,3)$. {\sf (Definition: power law exponent)} 

A \emph{(dynamic) network model} is a sequence of random graphs $(\cG_N)_{N\in\N}$ with the set of vertices of $\cG_N$ given
by $[N]:=\{1,2,\ldots,N\}$ and the set of unoriented edges of $\cG_N$ given by a random symmetric subset of $[N]\times [N]$. 
Occasionally we shall allow multiple edges between the same pair of vertices, but this has no bearing on the connectivity
problems discussed here, and is for convenience only. We write $v \leftrightarrow w$  if the vertices $v,w$ are connected 
by an edge in the graph $\mathcal G_N$. The graph distance is given by
$$d_N(v,w) := \min\big\{ n \colon \exists\, v=v_0, v_1, \ldots, v_n=w\in \cG_N 
\mbox{ such that } v_{i-1} \leftrightarrow v_i \,\,\forall\, 1\leq 1 \leq n\big\}.$$
The main aim of this paper is to provide techniques to find lower bounds on the \emph{typical distance},
i.e.\ the asymptotic graph distance of two randomly chosen vertices in the graph~$\cG_N$.
%\medskip
%
Our first result is based on the following assumption. 
%which typically holds for preferential attachment models, see the examples later in this section.
\medskip

%In preferential attachment models %with strong preference 
%one can typically verify the following assumption for an appropriate  parameter $0<\gamma<1$.\\[.2cm]

\noindent \emph{Assumption} $\mathrm{PA}(\gamma)$: \\[1mm]
\begin{shadedbox}
There exists $\kappa$ such that, for all  $N$ and pairwise distinct vertices $v_0,\dots,v_\ell\in[N]$,
$$\mathbb{P}\big\{ v_0\leftrightarrow v_1\leftrightarrow v_2\leftrightarrow \cdots \leftrightarrow v_\ell \big\}
\leq   \,\prod_{k=1}^{\ell} \kappa\, (v_{k-1}\wedge v_k)^{-\gamma} \,(v_{k-1}\vee v_k)^{\gamma-1} .$$
\end{shadedbox}
\medskip

In preferential attachment models with power law exponent $\tau$, Assumption~$\mathrm{PA}(\gamma)$ is 
typically satisfied for all $\gamma>(\tau-1)^{-1}$. Hence we expect these networks to be 
ultrasmall if and only if $\frac12<\gamma<1$. Theorem~\ref{thm:PA}, our main result, gives a lower bound on the typical distance 
in this case.
\medskip

\begin{theorem}\label{thm:PA}
	Let $(\mathcal{G}_N)_{N\in\N}$ be a dynamic network model that satisfies Assumption~$\mathrm{PA}(\gamma)$ for 
	some $\gamma$ satisfying $\frac12<\gamma<1$, then, for random vertices $V$ and $W$ chosen independently and uniformly 
        from $[N]$, we have
	    $$ d_N(V,W)\geq 4 \,\frac{\log\log N}{\log(\frac{\gamma}{1-\gamma})} + {\mathcal O}(1) $$
	with high probability.
\end{theorem}

%\begin{rem}
%{\rm Assumption $\mathrm{PA}(\gamma)$ is typically verified in two steps. First one shows that there exists a constant 
%$\kappa_1$ such that, for $1\le v<w$,
%\begin{equation}
%\label{as:PA1}
%\P\{v \leftrightarrow w\}\le \kappa_1 v^{-\gamma} w^{\gamma-1}
%\end{equation}
%and, in the second step, one proves that there exists a constant $\kappa_2$ such that, for all $l\in\IN$ and all pairwise distinct %$v_0,\dots,v_l\in\IN$,  
%\begin{equation}
%\label{as:pathcor}
%\mathbb{P}\big\{v_0\leftrightarrow v_1\leftrightarrow v_2\leftrightarrow \dots \leftrightarrow v_\ell\big\}
%\leq  \prod_{k=1}^{\ell} \kappa_2\,\mathbb{P}\{v_{k-1}\leftrightarrow v_{k}\}.
%\end{equation}}
%\end{rem}

Examples of network models, in which Theorem~\ref{thm:PA} can be applied, will be given as Examples~1--2 in Section~3. 
They comprise various preferential attachment models with power law exponent $\tau\in(2,3)$. 
In all these cases Assumption~$\mathrm{PA}(\gamma)$ is satisfied for all $\gamma>(\tau-1)^{-1}$, and the theorem implies that  
$$d_N(V,W)\geq\big(4+o(1)\big)\frac{\log\log N}{-\log(\tau-2)}, \ \ \text{ with high probability as $N\to\infty$.}$$ 
Matching upper bounds are known from the literature.
\medskip

An approach similar to the above can be used to study lower bounds for the typical distance
of ultrasmall \emph{configuration networks}. In this class the connection probabilities look different and we 
have to formulate a different assumption.
\medskip

\pagebreak[3]

\noindent \emph{Assumption} $\mathrm{CM}(\gamma)$: \\[1mm]
\begin{shadedbox}
\label{as:noPA}
	There exists $\kappa$ such that, for all  $N$ and pairwise distinct vertices $v_0,\dots,v_\ell\in[N]$,
	$$\mathbb{P}\big\{ v_0\leftrightarrow v_1\leftrightarrow v_2\leftrightarrow \cdots \leftrightarrow v_\ell \big\}
\leq   \,\prod_{k=1}^{\ell} \kappa\, v_{k-1}^{-\gamma} \,v_k^{-\gamma}\, N^{2\gamma-1} .$$
\end{shadedbox}
\medskip

In configuration models with power law exponent $\tau$, Assumption~$\mathrm{CM}(\gamma)$ is 
typically satisfied for all $\gamma>(\tau-1)^{-1}$, and again we expect these networks to be 
ultrasmall if $\frac12<\gamma<1$.% 
\medskip%

\begin{theorem}\label{thm:noPA}
Let $(\mathcal{G}_N)_{N\in\N}$ be a dynamic network model that satisfies Assumption~$\mathrm{CM}(\gamma)$ for 
	some $\gamma$ satisfying $\frac12<\gamma<1$, then, for random vertices $V$ and $W$ chosen independently and uniformly 
        from $[N]$, we have
	$$d_N(V,W)\geq 2\, \frac{\log\log N}{\log(\frac{\gamma}{1-\gamma})} + {\mathcal O}(1), \ \ 
        \text{ with high probability as $N\to\infty$.}$$
\end{theorem}
\medskip

%\begin{rem}
%{\rm Assumption $\mathrm{CM}(\gamma)$ is typically verified in two steps. One checks that
%there exists a constant $\kappa_1$ such that, for $1\le v<w\le N$,
%\begin{equation}\label{as:cm}
%\P\{v \leftrightarrow w\}\le \kappa_1 \, v^{-\gamma} w^{-\gamma} N^{2\gamma-1}.
%\end{equation}
%and then the bounded correlation condition
%\begin{equation}
%\label{as:pathcor}
%\mathbb{P}\big\{v_0\leftrightarrow v_1\leftrightarrow v_2\leftrightarrow \dots \leftrightarrow v_\ell\big\}
%\leq  \prod_{k=1}^{\ell} \kappa_2\,\mathbb{P}\{v_{k-1}\leftrightarrow v_{k}\}.
%\end{equation}}
%\end{rem}
%\medskip

Examples of network models, in which Theorem~\ref{thm:noPA} can be applied, will be given as Examples~3--5 in Section~3. 
They comprise a variety of configuration models with power law exponent $\tau\in(2,3)$. 
% which are not constructed from the preferential attachment paradigm. 
In all these cases Assumption~$\mathrm{CM}(\gamma)$ is satisfied for all $\gamma>(\tau-1)^{-1}$, and the theorem implies that  
$$d_N(V,W)\geq\big(2+o(1)\big)\frac{\log\log N}{-\log(\tau-2)}, \ \ \text{ with high probability as $N\to\infty$.}$$ 
Again, in all examples matching upper bounds are known from the literature.
\medskip

The proof of both theorems is based on a \emph{constrained} or \emph{truncated first order method}, which we now briefly explain. 
% in the (more involved) case of Theorem~\ref{thm:PA}.  
We start with an explanation of the (unconstrained) first moment bound and its shortcomings.  
Let $v$, $w$ be distinct vertices of $\cG_N$. Then, for $\delta\in\N$,
 \begin{align*}
\P\{d_N(v,w)\le 2\delta\} &=\P\Big( \bigcup_{k=1}^{2\delta} \bigcup_{(v_1,\dots, v_{k-1})} \{v \leftrightarrow v_1 \leftrightarrow 
\dots \leftrightarrow v_{k-1}\leftrightarrow w\}\Big)\\
&\le  \sum_{k=1}^{2\delta} \sum_{(v_1,\dots, v_{k-1})} \prod_{j=1}^{k} p({v_{j-1},v_j}),
\end{align*}
where $(v_0,\dots, v_{k})$ is any collection of pairwise distinct vertices in $\cG_N$ with $v_0=v$ and $v_{k}=w$
and, for $m,n\in\N$,
$${p}(m,n):=\left\{ \begin{array}{ll}
\kappa (m\wedge n)^{-\gamma} (m\vee n)^{\gamma-1} & \mbox{if $\mathrm{PA}(\gamma)$ holds;}\\
\kappa\, m^{-\gamma} \, n^{-\gamma}\, N^{2\gamma-1}& \mbox{if $\mathrm{CM}(\gamma)$ holds.}\\
\end{array}\right.$$
Note that one can assign each path $(v_0,\dots,v_k)$ the weight
\begin{equation}\label{weight}
p(v_0,\dots,v_k):=\prod_{j=1}^k p(v_{j-1},v_j),
\end{equation}
and the upper bound is just the sum over the weights of all paths from~$v$ to~$w$ of length no more than $2\delta$. 
The shortcoming of this bound is that the paths that contribute most to the total weight are those that connect 
$v$, resp.\ $w$, quickly to vertices with extremely small indices. Since these are typically 
not present in the network, such paths have to be removed in order to get a reasonable estimate.
\medskip

\pagebreak[3]

To this end we define a decreasing sequence $(\ell_k)_{k=0,\dots, \delta}$ of positive integers and consider a tuple of vertices $(v_0,\dots,v_n)$ 
as \emph{admissible}  if $v_k\wedge v_{n-k}\ge \ell_k$ for all $k\in\{0,\dots,\delta \wedge n\}$.
We denote by $A_k^\ssup{v}$ the event that there exists a path 
$v=v_0\leftrightarrow \cdots \leftrightarrow v_k$ in the network such that $v_0\ge \ell_0,\dots, v_{k-1}\ge \ell_{k-1}$, $v_k<\ell_k$, i.e. a path
that traverses the threshold after exactly $k$ steps.  For fixed vertices $v,w\geq \ell_0$, the truncated first moment estimate is 
\begin{align}\label{tfme}
\P\{d_N(v,w)\le 2\delta\}\le \sum_{k=1}^\delta \P(A_k^\ssup{v}) +\sum_{k=1}^\delta \P(A_k^\ssup{w}) + 
\sum_{n=1}^{2\delta} \sum_{\heap{(v_0,\ldots,v_{n})}{\text{admissible}}} \P\big\{ v_0 \leftrightarrow \cdots \leftrightarrow v_n \big\},
\end{align}
where the admissible paths in the last sum start with $v_0=v$ and end with $v_n=w$.  
By assumption, 
$$\P\{v_0 \leftrightarrow \cdots \leftrightarrow v_n\} \le p (v_0,\ldots,v_n)$$
so that for $v\ge \ell_0$ and $k=1,\dots,\delta$,
\begin{equation}\label{majorant}
\P(A_k^{\ssup v}) \le \sum_{v_1=\ell_1}^N \dots  \sum_{v_{k-1}=\ell_{k-1}}^N \sum_{v_k=1}^{\ell_k-1 }  p(v,v_1,\dots,v_{k}).
\end{equation}
Given $\eps>0$ we choose $\ell_0=\lceil\eps N\rceil$ and $(\ell_j)_{j=0,\dots, k}$ decreasing fast enough so that the first two summands 
on the right hand side of~\eqref{tfme} together are no larger than $2\eps$.  For $k\in\{1,\dots,\delta\}$, set  
$$\mu_k^\ssup{v} (u):= \ind_{\{v\ge \ell_0\}} \sum_{v_1=\ell_1}^N\dots  \sum_{v_{k-1}=\ell_{k-1}}^N p({v,v_1,\dots,v_{k-1},u}),$$
and set $\mu_0^\ssup{v}(u)=\ind_{\{v=u\}}$. To rephrase the truncated moment estimate in terms of~$\mu$, note that $p$ is 
symmetric so that, for all $ n\le 2\delta$ and $n^*:=\lfloor n/2\rfloor$, 
\begin{align}
\sum_{\heap{(v_0,\ldots,v_{n})}{\text{admissible}}} \P\big\{ v_0 \leftrightarrow \cdots \leftrightarrow v_n \big\}
&\le  \sum_{v_1=\ell_1}^N\dots \sum_{v_{n^*}=\ell_{n^*}}^N \dots \sum _{v_{n-1}=\ell_1}^{N} p(v,\dots, v_{n^*}) p(v_{n^*},\dots,w)\notag\\
&=   \sum_{v_{n^*}=\ell_{n^*}}^N \mu_{n^*}^\ssup{v} (v_{n^*}) \mu_{n-n^*}^\ssup{w}(v_{n^*}).\label{thissum}
\end{align}
Using the recursive representation 
$$\mu_{k+1}^{\ssup v}(n)=\sum_{m=\ell_k}^N \mu_{k}^\ssup {v}(m) \,p(m,n)$$   
we establish upper bounds for $\mu_k^\ssup{v} (u)$, and use these to show that the rightmost term in~\eqref{tfme} remains small
if $\delta$ is chosen sufficiently small. Using the input from Assumptions~$\mathrm{PA}(\gamma)$, resp.~$\mathrm{CM}(\gamma)$, 
this will lead to the lower bounds for the typical distance in both theorems. Detailed proofs will be given in Section~4.

\section{Examples}

In this section we give five examples, corresponding to the best understood models of ultrasmall networks
in the mathematical literature. Examples~1--2 are of preferential attachment type and will be discussed using our main result, Theorem~\ref{thm:PA}, while Examples~3--5 are of
configuration type and will be discussed using Theorem~\ref{thm:noPA}.%
\medskip

\begin{example}[Preferential attachment with fixed outdegree]\label{exp:affpa}

This class of models is studied in the work of Hooghiemstra, van der Hofstad and coauthors. We base our discussion on the paper~\cite{DHH10},
where three qualitatively similar models are considered, see also \cite{H11} for a survey. We focus on the first model 
studied in \cite{DHH10},  which is most convenient to define, the two variants can be treated with the same method. The model depends
on two parameters, an integer $\emm\ge 1$ and a real $\del>-\emm$. %, where~$\emm$ is the number of edges added in every step. 
Roughly speaking,  in every step a new vertex  is added to the network and connected to $\emm$ existing vertices with a probability proportional to their
degree plus $\del$. Note that in the case $\emm=1$ the network has the metric structure of a tree, making this a degenerate case of less interest. 
The case famously studied by Bollob\'{a}s and Riordan~\cite{BR04} corresponds to $\del=0$ and $\emm\ge 2$ and leads to a network with $\tau=3$ and typical 
distance $\log N/\log \log N$, so that it lies outside the class of ultrasmall networks.
\smallskip

We first generate a dynamic network model $(\mathcal{G}_N)$ for the case $\emm=1$. By $Z[n,N]$,  $n\leq N$, 
we denote the degree of vertex~$n$ in  $\mathcal{G}_N$ (with the convention that self-loops add two
towards the degree of the vertex to which they are attached).
%of the random graph $\mathcal{G}_N=\mathcal{G}^{(1,a)_N}$ is as follows:
\begin{itemize}
\item $\mathcal G_1$ consists of a single vertex, labelled~$1$, with one self loop.
\item In each further step, given $\mathcal{G}_N$, we insert one new vertex, labelled $N+1$, and one new edge into the 
network such that the new edge connects the new vertex to vertex $m\in[N]$ with probability 
$$%P(m,N+1):=
\mathbb{P}\big\{m\leftrightarrow N+1\,\big|\, \mathcal{G}_N\big\}=\frac{Z[m,N]+\del}{N(2+\del)+1+\del},$$ 
or to itself with probability 
$$%P(N+1,N+1)=
\frac{1+\del}{N(2+\del)+1+\del}.$$
\end{itemize}
To generalise the model to arbitrary values of $\emm$, we take the graph $\mathcal{G}'_{\emm N}$ constructed using  parameters $\emm'=1$ and
$\del'=\del/\emm$, and merge vertices $\emm (k-1)+1,\dots,\emm k$ in the graph $\mathcal{G}'_{\emm N} $ into 
a single vertex denoted~$k$,  keeping all edges. 
%This creates multiple self loops and parallel edges, but these are of no 
%relevance for the issues we discuss here. 
We obtain asymptotic degree distributions which are power laws with 
exponent $\tau=3+\frac{\del}{\emm}$, so that we expect to be in the ultrasmall 
range if and only if $-\emm<\delta<0$.
\smallskip

\begin{prop}
For independent, uniformly chosen vertices $V$ and $W$ in the giant component of the preferential attachment model
with parameters $\emm\geq 2$ and $-\emm<\delta<0$ , we have
$$d_N(V,W) = (4+o(1)) \,  \frac{\log\log N}{-\log(1+\frac{\del}{\emm})}  \qquad \mbox{ with high probability. }$$
\end{prop}

\begin{rem}{\rm
The upper bound is proved in \cite{DHH10}, see the remark following Theorem~1.6. This paper leaves the 
problem of finding a lower bound open.  We resolve this problem by verifying Assumption $\mathrm{PA}(\gamma)$ 
for $\gamma=(2+\frac{\del}{\emm})^{-1}$ and applying Theorem~\ref{thm:PA}.}
\end{rem}

\begin{proof}
We look at $\emm=1$ first. In this case, we have, for $1\leq m< n\leq N$, 
\begin{equation}\label{copro}
\mathbb{P}\{m\leftrightarrow n\} =\frac{\E Z[m,n-1]+\del}{n(2+\del)-1}.
\end{equation}
It is easy to see that 
$$\E\big[Z[m,n]+\del \, \big| \, Z[m,n-1]\big]=
\big(Z[m,n-1]+\del\big) \, \frac{n(2+\del)}{n(2+\del)-1},$$ and hence 
$$\E\big[Z[m,n]+\del\big] 
= (1+\del)\, \frac{\Gamma(n+1)\Gamma(m-\frac1{2+\del})}{\Gamma(n+\frac{1+\del}{2+\del})\Gamma(m)}.$$
In particular there exist constants $0<c<C$ such that
$$c\, \Kl\frac{n}{m}\Kr^{\frac{1}{2+\del}} \leq \E Z[m,n] \leq C\, \Kl\frac{n}{m}\Kr^{\frac{1}{2+\del}} \qquad \mbox{ for all }1\leq m < n.$$ 
Combining this with~\eqref{copro} yields, for $\gamma=\frac{1}{2+\delta}$ and a suitable $\kappa_1>0$, that
\begin{equation}\label{ersterschritt}
\P\{ m \leftrightarrow n \} \leq \frac{C\,(n/m)^\gamma + \del}{n(2+\del)-1} 
\leq \kappa_1 n^{\gamma-1} m^{-\gamma}\quad \mbox{ for all }1\leq m < n.
\end{equation}
To verify $\mathrm{PA}(\gamma)$, following \cite[Lemma 2.1]{DHH10}
we find that for distinct vertices $v_0, \ldots, v_l$ all events of the form 
$\{v_{j-1} \leftrightarrow v_j \leftrightarrow v_{j+1}\}$ with $j\in\{1,\dots,l-1\}$ and $v_j<v_{j-1},v_{j+1}$, and all 
events  $\{v_{j-1}\leftrightarrow  v_j\}$ which are not part of these,  are nonpositively correlated, in the sense that the probability of
all of them occurring is smaller than the product of the probabilities.
Recalling also~\eqref{ersterschritt} it remains to show that for $m<v,w$,
\begin{equation}\label{zweiterschritt}
\mathbb{P}\{v \leftrightarrow m \leftrightarrow w\} \leq \kappa_2\, v^{\gamma-1} w^{\gamma-1} m^{-2\gamma},
\end{equation}
%C \, \mathbb{P}\{m \leftrightarrow v\} \mathbb{P}\{m \leftrightarrow w\},$$
for some finite constant $\kappa_2>0$. 
To this end we let $\{(Z^{\ssup{k,m}}_n)_{n\geq m} \colon  k,m\in\mathbb{N}\}$ 
denote the collection of right-continuous Markov jump processes starting at $Z^{\ssup{k,m}}_{m-}=k$, jumping instantly
at time $m$ and subsequently at integer time-steps following the rule 
$$\mathbb{P}\big\{Z^{\ssup{k,m}}_{n}=Z^{\ssup{k,m}}_{n-}+1\, \big|\, Z^{\ssup{k,m}}_{n-}\big\}=
\frac{Z^{\ssup{k,m}}_{n-}+\del}{n(2+\del)-\del}=1-\mathbb{P}\big\{Z^{\ssup{k,m}}_{n}=Z^{\ssup{k,m}}_{n-}\,\big|\,Z^{\ssup{k,m}}_{n-}\big\}.$$
Note that $(Z[m,n])_{n\ge m}=(Z^{\ssup{1,m}}_n)_{n\ge m}$ in law and that, for $m<n$, the event $\{m\leftrightarrow n\}$ corresponds to
$\{\Delta Z^{\ssup{k,m}}_n=1\}$, where we write $\Delta Z^{\ssup{k,m}}_n := Z^{\ssup{k,m}}_n-Z^{\ssup{k,m}}_{n-}$. 
Note also that $Z^{\ssup{k_0,m}}_n$ is stochastically dominated by $Z^{\ssup{k,m}}_n$ for $k\geq k_0$. Hence, for $m< n_1<n_2$,
\begin{align*}
\mathbb{E}\big[Z^{\ssup{2,m}}_{n_2}\,\big| \, \Delta Z^{\ssup{2,m}}_{n_1}=1\big] 	& = \sum_{j=2}^{m-n_2+2}\sum_{k=2}^{m-n_1+1}j
\, \mathbb{P}\{Z^{\ssup{2,m}}_{n_2}=j\,|\, Z^{\ssup{2,m}}_{n_1-}=k, \Delta Z^{\ssup{2,m}}_{n_1}=1\}\\
&\phantom{rubadubdubdidoodlangldidledoh!}\times\mathbb{P}\{Z^{\ssup{2,m}}_{n_1-}=k\,|\, \Delta Z^{\ssup{2,m}}_{n_1}=1\}\\
& \leq \sum_{j=2}^{m-n_2+2}\sum_{k=2}^{m-n_1+1}\frac{j\,\mathbb{P}\{Z^{\ssup{k+1,n_1}}_{n_2}=j\}
\,(k+\del)\,\mathbb{P}\{Z^{\ssup{2,m}}_{n_1-}=k\}}{(n_1(2+\del)+1+\del)\,\mathbb{P}\{\Delta Z^{\ssup{2,m}}_{n_1}=1\}}\\
	& = \sum_{k=2}^{m-n_1+1}\frac{(k+\del)\,\mathbb{P}\{Z^{\ssup{2,m}}_{n_1-}=k\}\,\mathbb{E}Z^{\ssup{k+1,n_1}}_{n_2}}
	{(n_1(2+\del)+1+\del)\,\mathbb{P}\{\Delta Z^{\ssup{2,m}}_{n_1}=1\}}.
\end{align*}
As in the derivation of~\eqref{ersterschritt} the expectation in the last line can be bounded from
above by $c_0(k+1)n_2^{\gamma}n_1^{-\gamma}$, for some $c_0>0$. Similarly, we obtain
$\mathbb{P}\{\Delta Z^{\ssup{2,m}}_{n_1}=1\}\geq c_1 n_1^{\gamma-1}m^{-\gamma}$ and 
$$\mathbb{E}\big[(Z^{\ssup{2,m}}_{n_1-})^2\big]\leq c_2 \, m^{-\frac{2}{2+\del}}n_1^{\frac{2}{2+\del}},$$
for further constants $c_1,c_2>0$. Summarising, we obtain
\begin{align*}
\mathbb{E}\big[Z^{\ssup{2,m}}_{n_2} \, \big| \,\Delta Z^{\ssup{2,m}}_{n_1}=1\big]
%&\leq\sum_{k=2}^{m-n_1+1}\frac{(k+\del)\mathbb{P}\{Z^{\ssup{2,m}}_{n_1-1}=k\}\kappa_0(k+1)n_2^{\gamma}n_1^{-\gamma}\kappa_1 n_1^{1-\gamma}m^{\gamma}}{n_1(2+\del)+1+\del}\\
\leq c_3 n_2^{\gamma}n_1^{-2\gamma}m^{\gamma}\sum_{k=2}^{m-n_1+1}k^2
\mathbb{P}\big\{Z^{\ssup{2,m}}_{n_1-}=k\big\} \leq c_4 \, n_2^{\gamma}m^{-\gamma},
\end{align*}
for some $c_3, c_4>0$, and this establishes~\eqref{zweiterschritt}.
Finally, passing from $\emm=1$ to general~$\emm$ can be achieved by a simple union bound. 
\end{proof} 
\end{example}
\pagebreak[3]

A different class of  preferential attachment models was introduced in \cite{DM09} and further studied in \cite{DM10}. Here a new vertex 
is connected to any existing vertex independently with a probability depending (possibly nonlinearly) on its degree. In this model the number 
of edges  created in every step is asymptotically Poisson distributed. 
\medskip

\begin{example}[Preferential attachment with variable outdegree]\label{exp:nolpa}
This model is studied in the work of Dereich, M\"orters and coauthors, see~\cite{DM11} for a survey. The model depends
on a concave function $f\colon \N\cup\{0\} \to (0,\infty)$, which is called the \emph{attachment rule}. Roughly speaking, in 
every step a new vertex is added to the network and oriented edges from the new vertex to existing vertices are introduced
independently with a probability proportional to the current degree of the existing vertex. 
\smallskip

More precisely, to generate a dynamic network model $(\mathcal{G}_N)$ we assume that $f$ satisfies $f(0)\leq 1$ and
$f(1)-f(0)<1$. An important parameter derived from $f$ is the limit
$$\gamma:= \lim_{n\to \infty} \frac{f(n)}{n},$$
which always exists with $0\leq \gamma<1$, by concavity.
%To construct the networks ${\mathcal G}_N$ it is convenient to assume all edges as ordered, pointing from
%the younger to the older vertex.  
By $Z[n,N]$,  $n\leq N$, we denote %the indegree of vertex~$n$ in  $\mathcal{G}_N$, or equivalently 
the number of \emph{younger} vertices to which vertex~$n$ is connected in $\mathcal G_N$.
%of the random graph $\mathcal{G}_N=\mathcal{G}^{(1,a)_N}$ is as follows:
\begin{itemize}
\item $\mathcal G_1$ consists of a single vertex, labelled~$1$, and no edges.
\item In the $(N+1)$st step, given $\mathcal{G}_N$, we insert one new vertex, labelled $N+1$, and independently
for any $m\in[N]$ we introduce an edge from $N+1$ to $m$ with probability $$\frac{f(Z[m,N])}{N}.$$
\end{itemize}
By \cite[Theorem 1.1(b)]{DM09} the conditional distribution given $\mathcal G_N$ of the
number of edges  created in the $(N+1)$st step converges to a Poisson distribution and 
the empirical distribution of the degrees converges to a power law with exponent $\tau=1+\frac{1}{\gamma}$,
or more precisely to a random probability vector $(\mu_k)$ satisfying
$$\lim_{k\to\infty}\frac{\log \mu_k}{\log k}= 1+\frac{1}{\gamma}.$$
We therefore expect the network to be ultrasmall if and only if $\gamma>\frac12$.
\smallskip

\begin{prop}
For independent, uniformly chosen vertices $V$ and $W$ in the giant component of the preferential attachment model
with attachment rule~$f$ and derived parameter $\gamma>\frac12$, we have
$$d_N(V,W) = (4+o(1)) \,  \frac{\log\log N}{\log(\frac\gamma{1-\gamma})}  \qquad \mbox{ with high probability. }$$
\end{prop}

\begin{rem}{\rm
The upper bound can be proved by adapting the argument of~\cite{DHH10}, see the forthcoming thesis~\cite{M12} %or~\cite{DMM10} 
for details.  For the lower bound we verify Assumption $\mathrm{PA}(\gamma+\eps)$, for any $\eps>0$, and apply Theorem~\ref{thm:PA}.}
\end{rem}

\begin{proof} 
We first note that, for $v<w\in[N]$,
%\begin{equation}\label{eq:linf}
$$\mathbb{P}\{v\leftrightarrow w\}=\frac{\mathbb{E}f(Z[v,w-1])}{w-1}.$$ 
%\end{equation}
To estimate the expectation we note that by concavity, given $\eps>0$ there exists
$k$ such that, for all $n\geq k$, we have  $f(n)\leq f(k)+(\gamma+\eps) (n-k)$. An
easy calculation (see \cite[Lemma~2.7]{DM10}) shows that 
\begin{equation}\label{help}
\mathbb{E}f(Z[v,w-1])\leq C_1 w^{\gamma+\eps}v^{-\gamma-\eps}
\quad\mbox{for a suitable constant $C_1>0$.}
\end{equation}
We now use \eqref{help} to verify $\mathrm{PA}(\gamma+\eps)$. For $v< w \in [N]$, all events $\{v\leftrightarrow w\}$  
with different values of $v$ are independent. Hence $\P\{v_0 \leftrightarrow \cdots  \leftrightarrow  v_n\}$
can be decomposed into factors of the form $\mathbb{P}\{v_{j-1}\leftrightarrow v_j\leftrightarrow v_{j+1}\}$ with $v_j<v_{j-1},v_{j+1}$ and factors of the form
$\mathbb{P}\{v_{j-1} \leftrightarrow v_{j}\}$ for the remaining edges. It remains to estimate factors of the latter form. We may assume $v<u<w$ and get
$$\mathbb{P}\{u \leftrightarrow v\leftrightarrow w\} = \frac{\mathbb{E}[f(Z[v,u-1])f(Z[v,w-1])] }{(u-1)(w-1)}.$$ 
Arguing as in the derivation of \eqref{help} we get, for a suitable constant $C_2>0$,
$$\mathbb{E}\big[f(Z[v,w-1]) \, \big| \, Z[v,u-1] =k \big]  \leq C_2\, f(k)  w^{\gamma+\eps}u^{-\gamma-\eps}.$$
Hence 
$$\mathbb{E}\big[f(Z[v,u-1])f(Z[v,w-1])\big] \leq C_2\, \mathbb{E}\big[f(Z[v,u-1])^2\big]w^{\gamma+\eps}u^{-\gamma-\eps},$$
and, using a similar argument as above, we obtain $C_3>0$ such that
$$\mathbb{E}\big[f(Z[v,u-1])^2\big] \leq C_3 u^{2\gamma+\eps}v^{-2\gamma-\eps}.$$
Summarising, we obtain a constant $C_4>0$ such that
$$\mathbb{P}\{u \leftrightarrow v\leftrightarrow w\} \leq
C_4 u^{\gamma-1+\eps}v^{2\gamma-\eps} w^{\gamma-1+\eps},$$
as required to complete the proof.
\end{proof}
\end{example}
\medskip

We now give three examples of random networks in the universality class of configuration models. The first two belong to the
wide class of inhomogeneous random graphs, whose essential feature is the independence between different edges.
\medskip

\begin{example}[Expected degree random graph]\label{exp:irgdw}
This model is studied in the work of Chung and Lu, see~\cite{CL02} or~\cite{ChLu06}  for a survey. 
In its general form the model depends on a triangular scheme $w^{\ssup N}_1, \ldots, w^{\ssup N}_N$ of positive weights, 
where the weight $w^{\ssup N}_i$ plays the role of the expected degree of vertex $i$ in $\mathcal G_N$. 
The model is defined by the following two requirements:
\begin{itemize}
\item for every pair $(i,j)$ with $1\leq i\not=j \leq N$ the events $\{i \leftrightarrow j\}$ are independent, 
\item for every pair $(i,j)$ with $1\leq i\not=j \leq N$ we have 
$$\P\{ i \leftrightarrow j \} %= p_{ij} :
= \frac{w^{\ssup N}_i w^{\ssup N}_j}{\ell_N} \wedge 1 , \qquad
\mbox{ where } \ell_N:= \sum_{i=1}^N w^{\ssup N}_i.$$ 
\end{itemize}

\begin{prop}
For independent, uniformly chosen vertices $V$ and $W$ in the expected degree random graph with weights satisfying
$$c \, \big( \sfrac{N}{i}\big)^\gamma \leq w^{\ssup N}_i \leq C  \, \big( \sfrac{N}{i}\big)^\gamma \quad \mbox{ for all $1\leq i \leq N,$}$$
for some $\gamma>\frac12$ and constants $0<c\leq C$, we have
$$d_N(V,W) = (2+o(1)) \,  \frac{\log\log N}{\log(\frac\gamma{1-\gamma})}  \qquad \mbox{ with high probability. }$$
\end{prop}

\begin{proof}{}
The upper bound is sketched in~\cite{CL02}. For the lower bound we have to check Assumption~$\mathrm{CM}(\gamma)$.  
Note that, using the upper bound on the weights,
$$\P\{ i \leftrightarrow j \} \leq \sfrac{w^{\ssup N}_i w^{\ssup N}_j}{\ell_N} \leq C^2 \sfrac{N^{2\gamma}}{\ell_N} (ij)^{-\gamma}.$$ 
>From the lower bound on the  weights we get that $\ell_N\geq cN$, for some $c>0$, and hence
$\P\{ i \leftrightarrow j \} \leq \kappa N^{2\gamma-1} i^{-\gamma}j^{-\gamma}$ for a suitable $\kappa$.
Using the independence assumption we see that Condition~$\mathrm{CM}(\gamma)$ holds, and the lower bound follows from Theorem~2. 
\end{proof}
\end{example}

\begin{example}[Conditionally Poissonian random graph]\label{exp:norei}
This model is studied in the work of Norros and Reittu, see~\cite{NR06}. It is based on drawing an independent, identically
distributed sequence $\Lambda_1, \Lambda_2, \ldots$
of positive capacities. Conditional on this sequence, the dynamical network model is constructed as follows:

\begin{itemize}
\item $\mathcal G_1$ consists of a single vertex, labelled~$1$, and no edges.
\item In the $(N+1)$st step, given $\mathcal{G}_N$, we insert one new vertex, labelled $N+1$, and independently
for any $m\in[N]$ we introduce a random number of edges between $N+1$ and $m$ according to a Poisson distribution with
parameter $$\frac{\Lambda_i \Lambda_{N+1}}{L_{N+1}} \qquad \mbox{ for } L_{n}:=\sum_{k=1}^{n} \Lambda_k.$$
\item We further remove each edge in $\mathcal G_N$ independently with probability
$1- {L_N}/{L_{N+1}}$, and thus obtain $\mathcal G_{N+1}$.
\end{itemize}
Recall that having possibly several edges between two vertices has no relevance for the typical distances
in the giant component. In order to be in the ultrasmall regime we require the law of the capacities to be
power laws with exponent $2<\tau<3$.

\begin{prop}
Assume that the capacities in the conditionally Poissonian random graph satisfy
$$\P\{\Lambda_1>x\}= x^{1-\tau}\,(c+o(1)) \quad \mbox{ for all sufficiently large } x,$$
where $2<\tau<3$ and $c>0$ is constant.
For independent, uniformly chosen vertices $V$ and $W$ in the giant component we have
$$d_N(V,W) = (2+o(1)) \,  \frac{\log\log N}{-\log(\tau-2)}  \qquad \mbox{ with high probability. }$$
\end{prop}

\begin{rem}{\rm
The upper bound is proved in~\cite[Theorem 4.2]{NR06}, where it is also shown that a giant component exists.
For the lower bound we verify Assumption~$\mathrm{CM}(\gamma)$ for  $\gamma=1/(\tau-1)$ and apply 
Theorem~\ref{thm:noPA}.}
\end{rem}

\begin{proof}
We check that Assumption~$\mathrm{CM}(\gamma)$ holds with high probability, conditionally given the capacities. 
For fixed~$N$ we
put the capacities in decreasing order
$$\Lambda_N^{\ssup 1}> \Lambda_N^{\ssup 2} > \cdots > \Lambda_N^{\ssup N}$$
and relabel the vertices so that the $j$th vertex has weight~$\Lambda^{\ssup j}_N$. We recall from~\cite[Proposition 2.1]{NR06}
that the number of edges between vertices $i$ and $j$ in $\mathcal G_N$ is Poisson distributed with parameter
${\Lambda^{\ssup i}_N \Lambda^{\ssup j}_N}/L_N.$
As the edges are conditionally independent we only have to verify that, given $\eps>0$ there exists $\kappa>0$ such that
\begin{equation}\label{tosh}
1-\exp\Big(-\sfrac{\Lambda^{\ssup i}_N \Lambda^{\ssup j}_N}{L_N}\Big) 
\leq \kappa N^{2\gamma-1} i^{-\gamma} j^{-\gamma} \quad \mbox{ for all }1\leq i<j\leq N,
\end{equation}
with probability $\ge 1-2\eps$. By the law of large numbers $L_N$  is of order~$N$, 
so that it suffices to establish
$\Lambda^{\ssup i}_N \leq \kappa \, ({N}/i)^{\gamma}$ for all $1\leq i \leq N$.
To this end we denote by $S_N^{\ssup i}$ the number of potential values exceeding
$\kappa \, ({N}/i)^{\gamma}$. The random variable $S_N^{\ssup i}$ is binomially distributed
with parameters $N$ and $p:= \P\{ \Lambda_1> \kappa \, ({N}/{i})^{\gamma}\}\leq c(\kappa)\, \frac{i}{N},$
where $c(\varkappa)\downarrow 0$ for $\varkappa\uparrow\infty$. By Bernstein's inequality, see e.g.~\cite[(8)]{Be62},
$$\P\big\{ S^{\ssup i}_N > 2i\big\}  \leq \exp\bigg[\frac{-i^2/2}{{\mathrm{Var}}(S_N^{\ssup i})+i/3} \bigg]
\leq e^{-\frac3{8}\,i} \quad \mbox{ if $c(\kappa)<1$.}$$ 
Hence we may choose $M$ large enough so that $\sum_{i=M}^\infty \exp(-\frac3{8}i)<\eps$, ensuring that
with probability exceeding $1-\eps$ we have $\Lambda^{\ssup {2i}}_{N} \leq \kappa \, ({N}/{i})^{\gamma}$ 
for all $i\geq M$. It remains to give bounds on $\Lambda^{\ssup 1}_N, \ldots, 
\Lambda^{\ssup {2M}}_N$. By a standard Poisson approximation result, see e.g.~\cite[Proposition 3.21]{Re08},
we note that for any $1\leq i \leq 2M$, we have that $S_N^{\ssup i}$ converges weakly to a Poisson distribution
with parameter $\lambda:=\lim_{N\to\infty}
N \P\{ \Lambda_1> \kappa \, ({N}/{i})^{\gamma}\}\leq 2c(\kappa) M,$
and hence, by choosing $\kappa$ large, we can ensure that for large~$N$, we have
$\sum_{i=1}^{2M} \P\{ S^{\ssup i}_N > i\}\leq \eps,$  which completes the proof.
\end{proof}
\end{example}
\medskip

A model which also falls in the universality class of configuration models are the random networks with fixed 
degree sequence\footnote{In fact, in the literature these models are often called configuration models. We prefer 
to use the term for the wider class of models where vertices are equipped with an a-priori configuration of individual
features.}. This model is well studied and very detailed results on average distances in the case of
power laws with exponent~$\tau\in(2,3)$ are obtained, in particular by van der Hofstad et al.~in \cite{HHZ07}.
\medskip

\begin{example}[Random networks with fixed degree sequence]\label{exp:cm}
The idea behind this class of models is to enforce a particular power-law exponent by fixing the degree sequence of
the network in a first step. We therefore choose a sequence $D_1, D_2, \ldots$ of independent and identically distributed
random variables with values in the nonnegative integers. For given $N$ we assume that 
$$L_N:=\sum_{j=1}^N D_j$$
is even, which may be achieved by replacing $D_N$ by $D_{N}-1$ if necessary. Thus given
$D_1,\ldots, D_N$ we construct the network $\mathcal G_N$ as follows:
\begin{itemize}
\item To any vertex $m\in[N]$ we attach $D_m$ half-edges or stubs. 
\item The $L_N$ stubs are given an (arbitrary) order.
\item We start by pairing the first stub with a (uniformly) randomly chosen other stub, and continue pairing the lowest numbered 
unpaired stub with a remaining randomly chosen stub until all stubs are matched.
\item Any pair of stubs are connect to form an edge. 
\end{itemize}
Obviously the resulting network can have self-loop and double edges, but this has no relevance for the typical distances
in the giant component. In order to be in the ultrasmall regime we require the law of the degrees to be
a power law with exponent $2<\tau<3$.
\smallskip

\begin{prop}
Assume that there exists $c>0$ such that
$$\P\{D_1>x\}=x^{1-\tau}\,(c+o(1)) \quad \mbox{ for all sufficiently large } x.$$
For independent, uniformly chosen vertices $V$ and $W$ in the giant component we have
$$d_N(V,W) = (2+o(1)) \,  \frac{\log\log N}{-\log(\tau-2)}  \qquad \mbox{ with high probability. }$$
\end{prop}

\begin{rem}{\rm
This and much more is proved in~\cite[Theorem 1.2]{HHZ07}. For an alternative approach to the lower bound we 
now verify Assumption $\mathrm{CM}(\gamma)$ for any $\gamma<1/(\tau-1)$ and paths of length up to $\ell=\mathcal{O}(\log\log N)$, which is 
clearly sufficient to apply Theorem~\ref{thm:noPA}.}
\end{rem}

\begin{proof}
We observe that, given $D_1,\dots,D_N$, for pairwise disjoint vertices $v_1,\dots,v_\ell,v_{\ell +1}$, 
$$\mathbb{P}\big\{v_{\ell} \leftrightarrow v_{\ell+1} \,\big|\,v_1\leftrightarrow v_2\leftrightarrow\dots\leftrightarrow v_{\ell-1}\leftrightarrow v_\ell\big\} 
\leq \frac{D_{v_{\ell}}D_{v_{\ell+1}}}{L_N-2 \sum_{k=1}^\ell D_{v_k}},$$
where the denominator is a rough lower bound on the number of stubs unaffected by the conditioning event.
In particular, $\mathbb{P}\{i\leftrightarrow j\} \leq \sfrac{D_iD_j}{L_N-2D_i}$. Using the law of large numbers one can easily see that there 
is a $c>0$ such that 
$$L_N-2\sum_{k=1}^\ell D_{v_k}\geq cN \quad \mbox{  with high probability, }$$
 for any choice of $v_1,\dots,v_\ell$, if $\ell=\mathcal{O}(\log\log N)$. Therefore, to verify Assumption $\mathrm{CM}(\gamma)$ we only need to find appropriate bounds on the degrees of given vertices, which can be achieved (using the same relabeling) by a similar argument as in Example~4.\end{proof}
\end{example}

\section{Proofs}

\subsection{Proof of Theorem~\ref{thm:PA}}

In this  section, we assume validity of Assumption $\mathrm{PA}(\gamma)$ for a $\gamma\in(\frac12,1)$ with a fixed constant $\kappa$. 
Given a vector $(q(1), \ldots, q(n))$ we use the notation 
$$q[m]:=\sum_{i=1}^m q(i) \mbox{ for all $1\leq m\leq n$.} $$
We adopt the notation of the discussion at the end of Section~2. In particular recall the definition of $\mu_k^{\ssup v}$  and 
the key estimates~\eqref{tfme}, \eqref{majorant} and \eqref{thissum}, which combined give
\begin{equation}\label{key}
\P\{d_N(v,w)\le 2\delta\}\le \sum_{k=1}^\delta \mu_k^{\ssup v}[\ell_{k}-1] +\sum_{k=1}^\delta \mu_k^{\ssup w}[\ell_{k}-1] + 
\sum_{n=1}^{2\delta} \sum_{u=\ell_{n^*}}^N \mu_{n^*}^\ssup{v} (u) \mu_{n-n^*}^\ssup{w}(u).
\end{equation}
The remaining task of the proof is to choose $\delta\in\N$ and $2\leq \ell_\delta \leq \ldots \leq \ell_0\leq N$ which allow
the required estimates for the right hand side.  To do so we will make use of the recursive representation 
$$\mu_{k+1}^{\ssup v}(n)=\sum_{m=\ell_k}^N \mu_{k}^\ssup {v}(m) \,p(m,n) \quad \mbox{for $k\in\{0,\dots,\delta-1\}$ and $n\in[N]$, }$$
where $\mu_0^{\ssup v}(n)=\1\{v=n\}$ and $$p(m,n)=\kappa (m\wedge n)^{-\gamma} (m\vee n)^{\gamma-1}.$$
Denote by  $\bar \mu_k^\ssup{v}(m)=\ind_{\{m\ge \ell_k\}}\,\mu_k^\ssup{v}(m)$ the truncated version of $\mu_k^\ssup{v}$ and 
conceive $\mu_k^\ssup{v}$ and $\bar\mu^{\ssup v}_k$ as row vectors. Then 
\begin{align}\label{eq2511-1}
\mu_{k+1}^{\ssup v}=\bar \mu_{k}^\ssup {v} \,\bP_N,
\end{align}
where $\bP_N=(p(m,n))_{m,n=1,\dots,N}$. Our aim is to provide a majorant of the form
\begin{equation}\label{eq:mubound}
	\mu_k^{\ssup{v}}(m) \le \alpha_k  m^{-\gamma} + \one_{\{m>\ell_{k-1}\}} \beta_k m^{\gamma-1}
\end{equation}
for suitably chosen parameters~$\alpha_k,\beta_k\geq 0$. Key to this choice is the following lemma.

\begin{lemma}\label{lem:Pbound}
   Suppose that
 $2\leq\ell\leq N$, $\alpha, \beta\geq 0$ and $q\colon [N]\to[0,\infty)$ satisfies
	 $$
	 q(m)\le \ind\{m\geq \ell\} (\alpha m^{-\gamma} +\beta m^{\gamma-1}) \quad \mbox{ for all }m\in[N].
	 $$
Then there exists a constant $c>1$ (depending only on $\gamma$ and $\kappa$) such that
$$
q\bP_N(m) \le c 
\Bigl(\alpha \log \Bigl(\frac N \ell \Bigr)+\beta \,N^{2\gamma-1}\Bigr) m^{-\gamma}
+ \ind\{m>\ell\} c \Bigl(\alpha \ell^{1-2\gamma} +\beta \log \Bigl(\frac N \ell \Bigr)\Bigr)m^{\gamma-1}
$$
for all $m\in[N]$.
\end{lemma}

\begin{proof}
One has
	\begin{align*}
		 q\bP_N(m)& = \ind\{m>\ell\}  \sum_{k=\ell}^{m-1} q(k)\, p(k,m)+  \sum_{k=m\vee \ell}^{N} q(k)\, p(k,m) \\
&\le \ind\{m>\ell\}  \sum_{k=\ell}^{m-1} \kappa (\alpha k^{-\gamma} +\beta k^{\gamma-1}) k^{-\gamma} m^{\gamma-1} + \sum_{k=m\vee \ell}^{N}\kappa (\alpha k^{-\gamma} +\beta k^{\gamma-1}) k^{\gamma-1}  m^{-\gamma} \\
		          & \leq \kappa \Kl \alpha \sum_{k=m\vee \ell}^{N} k^{-1}+\beta  \sum_{k=m\vee\ell}^{N}k^{2\gamma-2}
	\Kr m^{-\gamma}\\		         
		 &\qquad\qquad+\ind\{m>\ell\} \kappa\Kl  \alpha  \sum_{k=\ell}^{m-1}k^{-2\gamma} + \beta \sum_{k=\ell}^{m-1}k^{-1}\Kr m^{\gamma-1} \\
		 & \leq  \kappa \Kl\alpha \log\Kl\frac{m}{\ell-1}\Kr+ \frac{\beta}{2\gamma-1}N^{2\gamma-1}\Kr m^{-\gamma}\\
		 &\qquad\qquad +\ind\{m>\ell\} \kappa \Kl\frac{\alpha}{1-2\gamma} (\ell-1)^{1-2\gamma}+\beta\log\Kl\frac{m}{\ell-1}\Kr\Kr m^{\gamma-1}.
	\end{align*}
This implies immediately the assertion since $\ell\ge 2$ by assumption. 
\end{proof}

We apply Lemma~\ref{lem:Pbound} iteratively. Fix $\eps>0$ small and start with 
$$\ell_0=\lceil \eps N\rceil, \, \alpha_1=\kappa (\eps N)^{\gamma-1} \mbox{ and } 
\beta_1=\kappa  (\eps N)^{-\gamma}.$$ Fix $v\ge\ell_0$. Then, for all $m\in[N]$, 
$$\begin{aligned}\mu_1^{\ssup v} (m) & = p(v,m) \le \kappa \ell_0^{\gamma-1} m^{-\gamma} + \one\{m>\ell_0\}\, \kappa \ell_0^{-\gamma} m^{\gamma-1}\\
& \le \alpha_1 m^{-\gamma} + \one\{m>\ell_0\}\, \beta_1 m^{\gamma-1}.
\end{aligned}$$
Now suppose, for some $k\in\N$, we have chosen $\alpha_k, \beta_k$ and an integer $\ell_{k-1}$ such that
$$\mu_k^{\ssup v} (m) \le \alpha_k m^{-\gamma} + \beta_k m^{\gamma-1} \mbox{ for all $m\in[N]$.}$$
We choose $\ell_k$ as an integer satisfying
\begin{equation}\label{elldef}
\frac{6\eps}{\pi^2k^2} \geq  \frac{1}{1-\gamma}  \alpha_{k} \ell_{k}^{1-\gamma},
\end{equation}
and assume $\ell_k\geq 2$. Pick $\alpha_k,\beta_k$  such that
\begin{equation}\label{eq112-2}
\begin{aligned}
\alpha_{k+1} & \geq c\, \Big( \alpha_k \log\Bigl(\sfrac N {\ell_k}\Bigr)+ \beta_k N^{2\gamma-1} \Big),\\
\beta_{k+1}  & \geq c\, \Big( \alpha_k \ell_k^{1-2\gamma} + \beta_k \log\Bigl(\sfrac N {\ell_k}\Bigr)  \Big).
\end{aligned}
\end{equation}
% \begin{pmatrix} \alpha_{k+1}\\ \beta_{k+1}\end{pmatrix} = c\, \begin{pmatrix} \log\Bigl(\frac N {\ell_k}\Bigr)  &N^{2\gamma-1}\\ \ell_k^{1-2\gamma} &    \log\Bigl(\frac N {\ell_k}\Bigr)   
% \end{pmatrix} \begin{pmatrix} \alpha_{k}\\ \beta_{k}\end{pmatrix}.
By the induction hypothesis we can apply Lemma~\ref{lem:Pbound} 
with $\ell=\ell_k$ and $q(m)=\bar \mu_k^{\ssup v}(m)$. Then, using~\eqref{eq2511-1}, 
\begin{equation}\label{major}
\mu_{k+1}^{\ssup v}(m) \leq \alpha_{k+1} m^{-\gamma} + \1\{m>\ell_k\} \beta_{k+1} m^{\gamma-1} \quad\mbox{for all $m\in[N]$,}
\end{equation}
showing that the induction can be carried forward up to the point where $\ell_k< 2$.
\smallskip

Summing over~\eqref{major} and using \eqref{elldef} we obtain
$$\mu^{\ssup v}_{k}[\ell_{k}-1] \le \frac{1}{1-\gamma}  \alpha_{k} \ell_{k}^{1-\gamma}
\leq \frac{6\eps}{\pi^2k^2}.$$
Hence the first two summands on the right hand side in~\eqref{key} are together smaller than $2\eps$. 
It remains to choose $\delta=\delta(N)$ as large as possible while ensuring that $\ell_\delta\geq 2$ and
$$\lim_{N\to \infty}
\sum_{n=1}^{2\delta} \sum_{u=\ell_{n^*}}^N \mu_{n^*}^\ssup{v} (u) \mu_{n-n^*}^\ssup{w}(u)=0.$$
To this end assume that $\ell_k$ is the largest integer satisfying~\eqref{elldef} and
the parameters~$\alpha_k, \beta_k$ are defined via equalities in \eqref{eq112-2}.  To establish lower bounds
for the decay of $\ell_k$ we investigate
%$$b_k:= \Big(\frac{\gamma}{1-\gamma}\Big)^{-k/2} \log \Big(\frac{N}{\ell_k}\Big).$$
the growth of $\eta_k:=N/\ell_k>0$. Going backwards through the definitions yields, for $k\geq 1$, that %$a_k:=\ell_k/N$ satisfies
$$\big( \eta_{k+2}^{-1} + \sfrac1N\big)^{\gamma-1} \leq 
\sfrac{c^2 (k+2)^2}{k^2} \eta_k^{\gamma}
+2c \,\sfrac{(k+2)^2}{(k+1)^2} \, \eta_{k+1}^{1-\gamma} \log \eta_{k+1},$$
with $\eta_1, \eta_2\leq C_0$ for some constant $C_0>0$ (which, as all constants in this paragraph, may depend 
on $\eps$).  It is easy to check inductively that for any solution of this system there exist constants $b,B>0$ such that,
\begin{equation}\label{ellbound}
\eta_k\leq  b\,\exp\Big( B \big(\sqrt{\sfrac{\gamma}{1-\gamma}} \big)^k \Big).
\end{equation}
We now use \eqref{major} to estimate
\begin{align*}
\sum_{n=1}^{2\delta} \sum_{u=\ell_{k}}^N \mu_{n^*}^\ssup{v} (u) \mu_{n-n^*}^\ssup{w}(u)
& \leq 2 \sum_{k=1}^{\delta} \sum_{u=\ell_{k}}^N \big(\alpha_ku^{-\gamma}+\beta_k u^{\gamma-1}\big)^2 \\
& \leq \sfrac4{2\gamma-1}\, \sum_{k=1}^{\delta} \big( \alpha_k^2 \ell_k^{1-2\gamma} + \beta_k^2 N^{2\gamma-1}\big)
 \leq \sfrac4{2\gamma-1}\, \delta  \big( \alpha_\delta^2 \ell_\delta^{1-2\gamma} + \beta_\delta^2 N^{2\gamma-1} \big) .
\end{align*}
Using~\eqref{elldef} and~\eqref{ellbound} the first summand in the bracket can be estimated by
$$\alpha_\delta^2 \ell_\delta^{1-2\gamma}
\leq \big(\delta^{-2} \sfrac{6\eps}{\pi^2} (1-\gamma)\big)^2 \ell_\delta^{-1} 
\leq \big(\sfrac{6\eps}{b\pi^2} (1-\gamma)\big)^2\, \frac1{N\delta^4} \,  
\exp\Big( B \,  \big(\sfrac{\gamma}{1-\gamma}\big)^{\delta/2}\Big).$$
Using equality in~\eqref{eq112-2} we get
$\beta_\delta \leq c ( \alpha_\delta \ell_\delta^{1-2\gamma} + 
\alpha_\delta N^{1-2\gamma} \log({N}/{\ell_\delta})).$
Noting that the second summand on the right hand side  is bounded by a multiple of the first, we find a constant
$C_1>0$ such that 
$\beta_\delta^2 N^{2\gamma-1} \leq C_1 \alpha_\delta^2 \ell_\delta^{1-2\gamma},$
and thus, for a suitable constant~$C_2>0$,
$$\sum_{n=1}^{2\delta} \sum_{u=\ell_{k}}^N \mu_{n^*}^\ssup{v} (u) \mu_{n-n^*}^\ssup{w}(u)
\leq C_2\,  \frac1{N\delta^3} \,  
\exp\Big( B \,  \big(\sfrac{\gamma}{1-\gamma}\big)^{\delta/2}\Big).$$
Hence,  for a suitable constant $C>0$, choosing
$$\delta \leq \frac{\log\log N}{\log\sqrt{\frac{\gamma}{1-\gamma}}}-C$$
we obtain that the term we consider goes to zero of order ${\mathcal O}((\log \log N)^{-3})$.
Note from~\eqref{ellbound} that this choice also ensures that $\ell_\delta\geq 2$. We have thus 
shown that 
$$\P\big\{ d_N(v,w) \geq 2 \delta\big\} \leq 2\eps + {\mathcal O}\big((\log \log N)^{-3}\big),$$
whenever $v,w \geq \ell_0=\lceil \eps N\rceil$, which implies the statement of Theorem~1.
\bigskip

\subsection{Proof of Theorem~\ref{thm:noPA}}
In this  section, we assume validity of Assumption $\mathrm{CM}(\gamma)$ for some $\gamma\in(\frac12,1)$ 
with a fixed constant $\kappa\geq 1$. Recall again  the notation and framework from the introductory chapter. 
We use the same approach as in the proof
of Theorem~\ref{thm:PA} but now we have to consider the matrix $\bP_N:=(p(m,n))_{m,n\in [N]}$ given by
\begin{equation} 
p(m,n):=\kappa m^{-\gamma}n^{-\gamma}N^{2\gamma-1} \text{ for } m,n\in [N]. \label{def:transprob2}
\end{equation}
We obtain the following lemma, which is the analogue of Lemma~\ref{lem:Pbound}.
\medskip
\begin{lemma}\label{lem:Pbound2}
	Suppose that $2 \leq \ell\leq N$ and $q\colon [N]\to[0,\infty)$ satisfies
	 $$	 q(m)\le \ind\{m\geq \ell\}\, m^{\gamma-1}\ell^{-\gamma}  \quad \mbox{ for all }m\in[N]. $$
Then, for all $m\in[N]$, 
	$$q\bP_N(m)\le \kappa \, m^{-\gamma} N^{\gamma-1}\, \bigl(\sfrac{N}{\ell}\bigr)^{\gamma}\log\big(\sfrac{N-1}{\ell-1}\big).$$
\end{lemma}

\begin{proof}
By \eqref{def:transprob2} and the assumption on $q$,
	\begin{align*}
	 q\bP_N(m)	=\sum_{i=1}^Nq(i)p(i,m)
	\leq \kappa m^{-\gamma}\ell^{-\gamma}N^{2\gamma-1}\sum_{i=\ell}^N \sfrac1{i}
  \leq \kappa m^{-\gamma}\ell^{-\gamma}N^{2\gamma-1}\,\log\big(\sfrac{N-1}{\ell-1}\big),
	\end{align*}
	which implies the statement of the lemma.
\end{proof}

For fixed $\eps>0$ we first construct inductively a strictly decreasing sequence of integers $(\ell_k)_{k=0, \ldots, \delta}$ 
by letting $\ell_0=\left\lceil\eps N\right\rceil$ and defining $\ell_{k+1}$ as the largest integer such that, given $\ell_k\geq 2$,
\begin{equation}\label{eq:noPAelldef}
	\frac{\kappa}{1-\gamma} \,\Kl\frac{\ell_{k+1}}{N}\Kr^{1-\gamma}\leq\frac{6\eps}{\pi^2(k+1)^2}\, \Big(\log\big( \sfrac{N-1}{\ell_{k}-1}\big)\Big)^{-1}\, \Kl\frac{\ell_k}{N}\Kr^{\gamma}.
\end{equation}
Recall the definition and recursive formula
for $\mu_k^{\ssup v}$ and let $\bar{\mu}_k^{\ssup v}(m):=\ind{\{m\geq\ell_k\}}\mu_k^{\ssup v}(m)$. Then $\mu_{k+1}^{\ssup v}(m)= \bar{\mu}_k^{\ssup v}\bP_N(m).$ We now apply inductively Lemma~\ref{lem:Pbound2} and obtain,
\begin{equation}\label{useful}
	\mu_{k}^{\ssup v}(m)\leq \kappa\,  m^{-\gamma} N^{\gamma-1}\Bigl(\frac{N}{\ell_{k-1}}\Bigr)^{\gamma} \log\Big(\sfrac{N-1}{\ell_{k-1}-1}\Big)
	\leq m^{-\gamma}\ell_{k}^{\gamma-1},\;\text{ for all } m\in[N].  
	\end{equation}
Note that the second inequality in~\eqref{useful} follows from \eqref{eq:noPAelldef}, and hence $\bar \mu_{k}^{\ssup v}(m) \leq  m^{\gamma-1}\ell_{k}^{-\gamma}$, which allows us to continue the induction. Considering the truncated first moment estimate~\eqref{tfme} for our choice of $(\ell_k)_{k=0, \ldots, \delta}$, we obtain from \eqref{useful} that %for the events $A^{\ssup v}_k$ 
$$\mathbb{P}\big(A^{\ssup v}_k\big)
\leq\mu_{k}^{\ssup v}[\ell_{k}-1]\leq \sfrac{\kappa}{1-\gamma} \, \Bigl(\frac{\ell_k}{N}\Bigr)^{1-\gamma}\Bigl(\frac{N}{\ell_{k-1}}\Bigr)^{\gamma} \, \log\Big(\sfrac{N-1}{\ell_{k-1}-1}\Big).$$
Hence~\eqref{eq:noPAelldef} entails that  $\sum_{k=1}^\del \mathbb{P}\big(A^{\ssup v}_k\big) \leq \eps$.
The last step is to choose $\delta=\delta(N)$ as large as possible while ensuring that $\ell_\delta\geq 2$ and
\begin{equation}\label{soll}
\lim_{N\to \infty}
\sum_{n=1}^{2\delta} \sum_{u=\ell_{n^*}}^N \mu_{n^*}^\ssup{v} (u) \mu_{n-n^*}^\ssup{w}(u)=0.
\end{equation}
By \eqref{useful} the term on the left can be bounded by a constant multiple of $N^{2\gamma-2} \sum_{k=1}^\del  \ell_k^{1-2\gamma}.$
To verify~\eqref{soll} we have to bound the growth of the values $\eta_k:=\frac{N}{\ell_k}$. 
The choice made in \eqref{eq:noPAelldef} implies that $(\eta_k)_{k\ge0}$ obeys $\eta_0\leq\eps^{-1}$ and $$\big(\eta_{k+1}^{-1}+\sfrac{1}{N}\big)^{\gamma-1}<\sfrac{\pi^2 \kappa}{1-\gamma} \, \sfrac{(k+1)^2}{6\eps} \,  \eta_k^{\gamma} \log(2\eta_k),\;\text{ for }k\geq0.$$ 
From this it is straightforward to verify inductively the existence of constants $b,B>0$, which only depend on $\eps, \kappa$ and $\gamma$, such that $$\eta_k\leq b \exp\Big( B \big(\sfrac{\gamma}{1-\gamma}\big)^k\Big),\;\text{ for }k\geq0.$$
Hence, we may choose a suitable constant $C>0$ such that for 
$$\del \leq \frac{\log\log N}{\log\Kl\frac{\gamma}{1-\gamma}\Kr}-C$$  
we have $\ell_\del\geq 2$. To complete the proof, we note that
$$N^{2\gamma-2} \sum_{k=1}^\del  \ell_k^{1-2\gamma}\leq
\frac 1N  \sum_{k=1}^\del  \eta_k^{2\gamma-1} 
\leq \del \, b \, N^{B\big(\sfrac{\gamma}{1-\gamma}\big)^{-C}-1},$$
which implies convergence in \eqref{soll} when $C$ is chosen large enough.
\bigskip

\noindent {\bf Acknowledgements:} The third author would like to acknowledge the support of \emph{EPSRC} through the
award of an  \emph{Advanced Research Fellowship}.
\bigskip

%\section{Conclusion}

%We have provided a model free approach to lower bounds for typical distances in ultrasmall random 
%networks. This approach shows that this class of networks can be divided into two different \emph{universaility 
%classes}. The class of configuration models and the class of preferential attachment models not only have typical 
%distance in the giant component different by a factor of two, but shortest paths in the networks also 
%have a different structure. It should be noted that universality results as above cannot be expected for the
%\emph{diameter} of an ultrasamll network. van der Hofstad and Hooghiemstra~\cite{HH08} observe a strong model dependence
%for the diameter of ultrasmall networks.

{
\bibliographystyle{alpha}
\bibliography{loglog}

\begin{thebibliography}{DHH10}

\bibitem[Ben62]{Be62}
G.~Bennett.
\newblock Probability inequalities for the sum of independent random variables.
\newblock {\em Journal of the American Statistical Association}, 57:33--45,
  1962.

\bibitem[BR04]{BR04}
B.~Bollob{\'a}s and O.~Riordan.
\newblock The diameter of a scale-free random graph.
\newblock {\em Combinatorica}, 24:5--34, 2004.

\bibitem[CH03]{CH03}
R.~Cohen and S.~Havlin.
\newblock Scale-free networks are ultrasmall.
\newblock {\em Physical Review Letters}, 90:058701, 2003.

\bibitem[CL03]{CL02}
F.~Chung and L.~Lu.
\newblock The average distance in a random graph with given expected degrees.
\newblock {\em Internet Mathematics}, 1:91--113, 2003.

\bibitem[CL06]{ChLu06}
F.~Chung and L.~Lu.
\newblock {\em Complex graphs and networks}, volume 107 of {\em CBMS Regional
  Conference Series in Mathematics}.
\newblock Published for the Conference Board of the Mathematical Sciences,
  Washington, DC, 2006.

\bibitem[DHH10]{DHH10}
S.~Dommers, R.~van~der Hofstad, and G.~Hooghiemstra.
\newblock Diameters in preferential attachment models.
\newblock {\em Journal of Statistical Physics}, 139:72--107, 2010.

\bibitem[DM09]{DM09}
S.~Dereich and P.~M{\"o}rters.
\newblock Random networks with sublinear preferential attachment: degree
  evolutions.
\newblock {\em Electronic Journal of Probability}, 14:1222--1267, 2009.

\bibitem[DM10]{DM10}
S.~Dereich and P.~M{\"o}rters.
\newblock Random networks with sublinear preferential attachment: the giant
  component.
\newblock {\em Preprint arXiv:1007.0899. Submitted for publication}, 2010.

\bibitem[DM11]{DM11}
S.~Dereich and P.~M{\"o}rters.
\newblock Random networks with concave preferential attachment rule.
\newblock {\em Jahresbericht der Deutschen Mathematiker Vereinigung}, 2011.

\bibitem[DMS03]{DMS03}
S.~N. Dorogovtsev, J.~F.~F. Mendes, and A.~N. Samukhin.
\newblock Metric strucure of random networks.
\newblock {\em Nuclear Physics B}, 653:307--338, 2003.

\bibitem[HH08]{HH08}
R.~van~der Hofstad and G.~Hooghiemstra.
\newblock Universality for distances in power-law random graphs.
\newblock {\em Journal of Mathematical Physics}, 49:125209, 2008.

\bibitem[HHZ07]{HHZ07}
R.~van~der Hofstad, G.~Hooghiemstra, and D.~Znamenski.
\newblock Distances in random graphs with finite mean and infinite variance
  degrees.
\newblock {\em Electronic Journal of Probability}, 12:703--766, 2007.

\bibitem[Hof10]{H11}
R.~van~der Hofstad.
\newblock Random graphs and complex networks.
\newblock 2010.

\bibitem[M{\"o}n12]{M12}
C.~M{\"o}nch.
\newblock {\em Distances in preferential attachment networks}.
\newblock PhD thesis, University of Bath, 2012.

\bibitem[NR06]{NR06}
I.~Norros and H.~Reittu.
\newblock On a conditionally {P}oissonian graph process.
\newblock {\em Advances in Applied Probability}, 38:59--75, 2006.

\bibitem[NR08]{NR08}
I.~Norros and H.~Reittu.
\newblock Network models with a `soft hierarchy': A random graph construction
  with loglog scalability.
\newblock {\em IEEE Network}, 22:40--46, 2008.

\bibitem[Res08]{Re08}
S.~I. Resnick.
\newblock {\em Extreme values, regular variation, and point processes}.
\newblock Springer, New York, 2008.

\bibitem[RN02]{RN02}
H.~Reittu and I.~Norros.
\newblock On the effect of very large nodes in internet graphs.
\newblock In {\em Globecom'02, {\rm Vol. III (Proc. Global Telecommunications
  Conf., Taipei, 2002)}}, pages 2624--2628, 2002.

\end{thebibliography}
}

\end{document}